\documentclass{amsart}%
\usepackage{amsfonts}
\usepackage{graphicx}
\usepackage{amsmath}
\usepackage{amssymb}
\usepackage{rotating}%
\newtheorem{theorem}{Theorem}

\newtheorem{corollary}[theorem]{Corollary}

\newtheorem{definition}[theorem]{Definition}

\newtheorem{lemma}[theorem]{Lemma}
\newtheorem{notation}[theorem]{Notation}

\newtheorem{proposition}[theorem]{Proposition}
\newtheorem{remark}[theorem]{Remark}

\begin{document}
\title[ ]{A Study on $q$-Appell Polynomials from Determinantal Point of View}
\author{Marzieh Eini Keleshteri and Naz{i}m I. Mahmudov}
\address{Mathematics Department, Eastern Mediterranean University, Famagusta, North
Cyprus, via Mersin 10, Turkey}
\email{marzieh.eini@emu.edu.tr, nazim.mahmudov@emu.edu.tr}
\subjclass{Special Functions}
\keywords{$q$-polynomials, Determinantal, Algebraic, Appell, Euler, Bernoulli, Genocchi}

\begin{abstract}
This research is aimed to give a determinantal definition for the $q$-Appell
polynomials and show some classical properties as well as find some
interesting properties of the mentioned polynomials in the light of the new definition.

\end{abstract}
\maketitle

\section{Introduction, preliminaries and definitions}

Throughout this research we always apply the following notations: $\mathbb{N}$
indicates the set of natural numbers, $\mathbb{N}_{0}$ indicates the set of
non-negative integers, $\mathbb{R}$ indicates set of all real numbers, and
$\mathbb{C}$ denotes the set of complex numbers. We refer the readers to
\cite{Andrews} for all the following $q$-standard notations. The $q$-shifted
factorial is defined as%
\[
(a;q)_{0}=1,\ (a;q)_{n}=\prod\limits_{j=0}^{n-1}(1-q^{j}a),\
n\in\mathbb{N}, \  (a;q)_{\infty}=\prod\limits_{j=0}^{\infty}(1-q^{j}%
a),\  |q|<1,a\in\mathbb{C}.
\]
The $q$-numbers and $q$-factorial are defined by%
\newline
\[\lbrack a]_{q}=\frac{1-q^{a}}{1-q}\ \ (q\neq1),\ \ \lbrack0]!=1,\quad\lbrack
n]_{q}!=[1]_{q}[2]_{q}\ldots\lbrack n]_{q},\quad n\in\mathbb{N},a\in\mathbb{C},
\]
respectively. The $q$-polynomial coefficient is defined by%
\[
\left[
\begin{array}
[c]{c}%
n\\
k
\end{array}
\right]  _{q}=\frac{[n]_{q}!}{[k]_{q}![n-k]_{q}!}.
\]
The $q$-analogue of the function $(x+y)^{n}$ is defined by%
\begin{equation}
(x+y)_{q}^{n}:=\sum\limits_{k=0}^{n}\left[
\begin{array}
[c]{c}%
n\\
k
\end{array}
\right]  _{q}q^{1/2k(k-1)}x^{n-k}y^{k},\quad n\in\mathbb{N}_{0}. \label{00}%
\end{equation}
The $q$-binomial formula is known as%
\[
(1-a)_{q}^{n}=\prod\limits_{j=0}^{n-1}(1-q^{j}a)=\sum\limits_{k=0}^{n}\left[
\begin{array}
[c]{c}%
n\\
k
\end{array}
\right]  _{q}q^{1/2k(k-1)}(-1)^{k}a^{k}.
\]
In the standard approach to the $q$-calculus two exponential functions are
used:%
\begin{align*}
e_{q}\left(  z\right)   &  =\sum_{n=0}^{\infty}\frac{z^{n}}{\left[  n\right]
_{q}!}=\prod_{k=0}^{\infty}\frac{1}{\left(  1-\left(  1-q\right)
q^{k}z\right)  },\ \ \ 0<\left\vert q\right\vert <1,\ \left\vert z\right\vert
<\frac{1}{\left\vert 1-q\right\vert },\ \ \ \ \ \ \ \\
E_{q}\left(  z\right)   &  =\sum_{n=0}^{\infty}\frac{q^{\frac{1}{2}n\left(
n-1\right)  }z^{n}}{\left[  n\right]  _{q}!}=\prod_{k=0}^{\infty}\left(
1+\left(  1-q\right)  q^{k}z\right)  ,\ \ \ \ \ \ \ 0<\left\vert q\right\vert
<1,\ z\in\mathbb{C}.\
\end{align*}
From this form we easily see that $e_{q}\left(  z\right)  E_{q}\left(
-z\right)  =1$. Moreover,%
\[
D_{q}e_{q}\left(  z\right)  =e_{q}\left(  z\right)  ,\ \ \ \ D_{q}E_{q}\left(
z\right)  =E_{q}\left(  qz\right)  ,
\]

The $q$-derivative of a function $f$ at point $0\neq z\in\mathbb{C}$ is
defined as%
\begin{equation}
D_{q}f\left(  z\right)  :=\frac{f\left(  qz\right)  -f\left(  z\right)
}{qz-z},\ \ \ \ 0<\left\vert q\right\vert <1. \label{q-der}%
\end{equation}

\begin{proposition}
\cite{Kac} Consider two arbitrary functions $f(z)$ and $g(z)$. The following
relations hold for the $q$-derivative:

\begin{enumerate}
\item[a)] if $f$ is differentiable, \ \
\[
\ \lim_{q\rightarrow1}D_{q}f\left(  z\right)  =\frac{df(z)}{dz},
\]
where $\frac{d}{dz}$ indicates the ordinary derivative is defined in Calculus.

\item[b)] $D_{q}$ is a linear operator; that is, for arbitrary constants $a$
and $b$%
\[
D_{q}(af(z)+bg(z))=aD_{q}(f(z))+bD_{q}(g(z)),
\]

\item[c)]
\[
D_{q}(f(z)g(z))=f(qz)D_{q}g(z)+g(z)D_{q}f(z),
\]

\item[d)]
\[
D_{q}(\frac{f(z)}{g(z)})=\frac{g(qz)D_{q}f(z)-f(qz)D_{q}g(z)}{g(z)g(qz)}.
\]

\end{enumerate}
\end{proposition}

For the first time in 1909 Jackson introduced the $q$-analogue of \ Taylor
series expansion of an arbitrary function $f(z)$ for $0<q<1$, as follows%

\begin{equation}
f(z)=\sum\limits_{n=0}^{\infty}\frac{(1-q)^{n}}{(q;q)_{n}}D_{q}^{n}%
f(a)(z-a)_{q}^{n}, \label{0}%
\end{equation}

where $D_{q}^{n}f(a)$ is the $n^{th}$ $q$-derivative of the function $f$ at
point $a.$

Furthermore, Jackson integral of an arbitrary function $f(x)$ is defined as,
\cite{Kac}%

\begin{equation}
\int f(x)d_{q}x=(1-q)\sum\limits_{n=0}^{\infty}xq^{j}f(xq^{j}),\quad 0<q<1. \label{int}%
\end{equation}

Appell polynomials for the first time were defined by Appell in 1880,
\cite{Appell}. Inspired by the work of Throne \cite{Throne}, Sheffer
\cite{Sheffer}, and Varma \cite{Varma}, Al-Salam, in 1967, introduced the family
of $q$-Appell polynomials $\{A_{n,q}(x)\}_{n=0}^{\infty}$, and studied some of
their properties \cite{Alsalam}. According to his definition, the n-degree
polynomials $A_{n,q}(x)$ are called $q$-Appell if they hold the following
$q$-differential equation%

\begin{equation}
D_{q,x}(A_{n,q}(x))=[n]_{q}A_{n-1,q}(x),\quad n=0,1,2,...\quad
\label{1}%
\end{equation}

Note to the fact that $A_{0,q}(x)$ is a non zero constant let say $A_{0,q}$.
To begin with the relation(\ref{1}) for $n=1$, i. e.%

\[
D_{q,x}(A_{1,q}(x))=[1]_{q}A_{0,q}(x)=A_{0,q}.
\]
Using Jackson integral for the $q$-differential equation above, we get%

\[
A_{1,q}(x)=A_{0,q}x+A_{1,q},
\]
where $A_{1,q}$ is an arbitrary constant. We can repeat the method above to obtain
$A_{2,q}(x),$ as below by starting from the property(\ref{1}) for $q$-Appell polynomials%

\[
D_{q,x}(A_{2,q}(x))=[2]_{q}A_{1,q}x=[2]_{q}A_{0,q}x+[2]_{q}A_{1,q}.
\]
Now take Jackson integral%

\[
A_{2,q}(x)=A_{0,q}x^{2}+[2]_{q}A_{1,q}+A_{2,q},
\]
where $A_{2,q\quad}$is an arbitrary constant.

By using induction on $n$ and applying similar method to the methods
used for finding $A_{1,q}(x)$, $A_{2,q}(x)$ and continuing taking Jackson
integrals we have%
\[
A_{n-1,q}(x)=A_{n-1,q}+\left[
\begin{array}
[c]{c}%
n-1\\
1
\end{array}
\right]  _{q}A_{n-2,q}x+\left[
\begin{array}
[c]{c}%
n-1\\
2
\end{array}
\right]  _{q}A_{n-3,q}x^{2}+...+A_{0,q}x^{n-1}.
\]
Considering the fact that for $ n=1,2,3,...,$ every $A_{n,q}(x)$ satisfies the relation (\ref{1}), we can write %
\begin{multline*}
D_{q,x}(A_{n,q}(x))=[n]_{q}A_{n-1,q}+[n]_{q}\left[
\begin{array}
[c]{c}%
n-1\\
1
\end{array}
\right]  _{q}A_{n-2,q}x \\
+[n]_{q}\left[
\begin{array}
[c]{c}%
n-1\\
2
\end{array}
\right]  _{q}A_{n-3,q}x^{2}+...+[n]_{q}A_{0,q}x^{n-1}.
\end{multline*}
Now, taking the Jackson integral of the $q$-differential equation above can lead
to%

\begin{multline*}
A_{n,q}(x)=A_{n,q}+[n]_{q}A_{n-1,q}x+\frac{[n]_{q}}{[2]_{q}}\left[
\begin{array}
[c]{c}%
n-1\\
1
\end{array}
\right]  _{q}A_{n-2,q}x^{2} \\
 +\frac{[n]_{q}}{[3]_{q}}\left[
\begin{array}
[c]{c}%
n-1\\
2
\end{array}
\right]  _{q}A_{n-3,q}x^{3}+...+\frac{[n]_{q}}{[n]_{q}}A_{0,q}x^{n},
\end{multline*}
where $A_{n,q}$ is an arbitrary constant. Since
\[
\frac{\lbrack n]_{q}}{[i]_{q}}\left[
\begin{array}
[c]{c}%
n-1\\
i-1
\end{array}
\right]  _{q}=\left[
\begin{array}
[c]{c}%
n\\
i
\end{array}
\right]  _{q},
\]
so for $n=0,1,2,...,$ we have%

\begin{equation}
A_{n,q}(x)=A_{n,q}+[n]_{q}A_{n-1,q}x+\left[
\begin{array}
[c]{c}%
n\\
2
\end{array}
\right]  _{q}A_{n-2,q}x^{2}+\left[
\begin{array}
[c]{c}%
n\\
3
\end{array}
\right]  _{q}A_{n-3,q}x^{3}+...+A_{0,q}x^{n}.
\label{1'}%
\end{equation}

It is worthy of note that according to the discussion above there exists a one
to one correspondence between the family of $q$-Appell polynomials
$\{A_{n,q}(x)\}_{n=0}^{\infty}$ and the numerical sequence $q$-Appell
polynomials $\{A_{n,q}\}_{n=0}^{\infty},$ \ \ $A_{n,q}\neq0.$ Moreover, every
$A_{n,q}(x)$ can be obtained recursively from $A_{n-1,q}(x)$ for
$n\geqslant1.$

Also, $q$-Appell polynomials can be defined by means of generating function
$A_{q}(t)$, as follows%

\begin{equation}
A_{q}(x,t):=A_{q}(t)e_{q}(tx)=\sum_{n=0}^{\infty}A_{n,q}(x)\frac{t^{n}%
}{\left[  n\right]  _{q}!},\quad 0<q<1, \label{2}%
\end{equation}

where
\begin{equation}
A_{q}(t):=\sum_{n=0}^{\infty}A_{n,q}\frac{t^{n}}{\left[  n\right]  _{q}%
!},\ \ A_{q}(t)\neq0, \label{3}%
\end{equation}

is an analytic function at $t=0$, $A_{n,q}(x):=A_{n,q}(0),$ and
$e_{q}(t)=\sum_{n=0}^{\infty}\frac{t^{n}}{\left[  n\right]  _{q}!}.$

Based on different selections for the generating function $A_{q}(t)$,
different families of $q$-Appell polynomials can be obtained. In the following
we mention some of them:

\begin{itemize}
\item[a)] Taking $A_{q}(t)=[1]_{q}=1$ leads to obtain the family including
all increasing integer powers of x starting from 0,%

\[
\{1,x,x^{2},x^{3},...\}.
\]

\item[b)] Taking $A_{q}(t)=\left(  \frac{t^{m}}{e_{q}\left(  t\right)
-T_{m-1,q}\left(  t\right)  }\right)  ^{\alpha}$, leads to obtain the family
of generalized $q$-Bernoulli polynomials $\mathfrak{B}_{n,q}^{[m-1,\alpha
]}(x,0)$, \cite{Mah}.

\item[c)] Taking $A_{q}(t)=\left(  \frac{2^{m}}{e_{q}\left(  t\right)
+T_{m-1,q}\left(  t\right)  }\right)  ^{\alpha}$, leads to obtain the family
of generalized $q$-Euler polynomials $\mathfrak{E}_{n,q}^{[m-1,\alpha]}(x,0)$,
\cite{Mah}.

\item[d)] Taking $A_{q}(t)=\left(  \frac{2^{m}t^{m}}{e_{q}\left(  t\right)
+T_{m-1,q}\left(  t\right)  }\right)  ^{\alpha}$, leads to obtain the family
of generalized $q$-Genocchi polynomials $\mathfrak{G}_{n,q}^{[m-1,\alpha
]}(x,0)$, \cite{Mah}.

\item[e)] Taking $A_{q}(t)=\left(  \frac{t^{m}}{\lambda e_{q}\left(  t\right)
-T_{m-1,q}\left(  t\right)  }\right)  ^{\alpha}$, leads to obtain the family
of generalized $q$-Apostol Bernoulli polynomials $B_{n,q}^{[m-1,\alpha
]}\left(  x,0;\lambda\right)  $ of order $\alpha$, \cite{Mah1}.

\item[f)] Taking $A_{q}(t)=\left(  \frac{2^{m}}{\lambda e_{q}\left(  t\right)
+T_{m-1,q}\left(  t\right)  }\right)  ^{\alpha}$, leads to obtain the family
of $q$-Apostol-Euler polynomials $E_{n,q}^{[m-1,\alpha]}\left(  x,0;\lambda
\right)  $ of order $\alpha$, \cite{Mah1}.

\item[g)] Taking $A_{q}(t)=\left(  \frac{2^{m}t^{m}}{\lambda e_{q}\left(
t\right)  +T_{m-1,q}\left(  t\right)  }\right)  ^{\alpha}$, leads to obtain
the family of $q$-Apostol-Genocchi polynomials $G_{n,q}^{[m-1,\alpha]}\left(
x,0;\lambda\right)  $ of order $\alpha$, \cite{Mah1}.

\item[h)] Taking $A_{q}(t)=H_{q}(t)=\sum\limits_{n=0}^{\infty}(-1)^{n}%
q^{n(n-1)}\frac{t^{2n}}{[2n]!!}$, leads to obtain the family of $q$-Hermite
polynomials $H_{n,q}\left(  x\right)  $.
\end{itemize}

Later, in 1982, Srivastava specified more characterizations of the family of
$q$-Appell polynomials, \cite{Sri}. Over the past decades, $q$-Appell
polynomials have been studied from different aspects in \cite{Sharma},
\cite{Ernst}, using different methods such as operator algebra their
properties are found in \cite{Lou}. Also, recently, the $q$-difference
equations satisfied by sequence of $q$-Appell polynomials have been derived by
Mahmudov, \cite{Mah3}. In this paper, inspired by the Costabile et al.'s
algebraic approach for defining Bernoulli polynomials as well as Appell
polynomials, for the first time, we introduce a determinantal definition of
the well known family of $q$-Appell polynomials, \cite{Costabile1},
\cite{Costabile2}. This new algebraic definition, not only allows us to
benefit from algebraic properties of determinant to prove the existing
properties of $q$-Appell polynomials more simpler, but also helps to find some
new properties. Moreover, this approach unifies all different families of
$q$-Appell polynomials some of which are mentioned in a)-h).

In the following sections, firstly we introduce the determinantal definition
of $q$-Appell polynomials and then we show that this definition matches with
the classical definitions. Next we prove some classical and new properties
related to this family in the light of the new definition and by using the
related algebraic approaches.
\section{$q$-polynomials from determinantal point of view}

Assume that $P_{n,q}(x)$ is an $n$-degree $q$-polynomial defined as follows%

\begin{equation}
\left\{
\begin{array}
[c]{l}%
P_{0,q}(x)=\frac{1}{\beta_{0}}\\
P_{n,q}(x)=\frac{(-1)^{n}}{(\beta_{0})^{n+1}}\left\vert
\begin{array}
[c]{ccccccc}%
1 & x & x^{2} & ... & ... & x^{n-1} & x^{n}\\
\beta_{0} & \beta_{1} & \beta_{2} & ... & ... & \beta_{n-1} & \beta_{n}\\
0 & \beta_{0} & \left[
\begin{array}
[c]{c}%
2\\
1
\end{array}
\right]  _{q}\beta_{1} & ... & ... & \left[
\begin{array}
[c]{c}%
n-1\\
1
\end{array}
\right]  _{q}\beta_{n-2} & \left[
\begin{array}
[c]{c}%
n\\
1
\end{array}
\right]  _{q}\beta_{n-1}\\
0 & 0 & \beta_{0} & ... & ... & \left[
\begin{array}
[c]{c}%
n-1\\
2
\end{array}
\right]  _{q}\beta_{n-3} & \left[
\begin{array}
[c]{c}%
n\\
2
\end{array}
\right]  _{q}\beta_{n-2}\\
\vdots &  &  & \ddots &  & \vdots & \vdots\\
\vdots &  &  &  & \ddots & \vdots & \vdots\\
0 & ... & ... & ... & 0 & \beta_{0} & \left[
\begin{array}
[c]{c}%
n\\
n-1
\end{array}
\right]  _{q}\beta_{1}%
\end{array}
\right\vert
\end{array}
\right.  , \label{4}%
\end{equation}
where $\beta_{0},\beta_{1},...,\beta_{n}\in\mathbb{R}\emph{,}$ $\beta_{0}%
\neq0,\ $ $n=1,2,3,...$

Then we can obtain the following results.

\begin{lemma}
\label{Lemma1}Suppose that $A_{n\times n}(x)$ is a matrix including elements
$a_{ij}(x)$ which are first order $q$-differentiable functions of variable
$x$. Then the $q$-derivative of $\det(A_{n\times n}(x))$ can be calculated by
the following formula.

\begin{multline}
D_{q,x}(\det(A_{n\times n}(x)))=D_{q,x}(\left\vert a_{ij}(x)\right\vert
)\\
=\sum_{i=1}^{n}\left\vert
\begin{tabular}
[c]{llll}%
$a_{11}(qx)$ & $a_{12}(qx)$ & $\ldots$ & $a_{1n}(qx)$\\
$\vdots$ & $\vdots$ & $\ddots$ & $\vdots$\\
$a_{i-1,1}(qx)$ & $a_{i-1,2}(qx)$ & $\ldots$ & $a_{i-1,n}(qx)$\\
$D_{q,x}(a_{i1}(x))$ & $D_{q,x}(a_{i2}(x))$ & $\ldots$ & $D_{q,x}(a_{in}%
(x))$\\
$a_{i+1,1}(x)$ & $a_{i+1,2}(x)$ & $\ldots$ & $a_{i+1,n}(x)$\\
$\vdots$ & $\vdots$ & $\ddots$ & $\vdots$\\
$a_{n1}(x)$ & $a_{n2}(x)$ & $\ldots$ & $a_{nn}(x)$%
\end{tabular}
\ \ \ \right\vert . \label{5}%
\end{multline}
\end{lemma}

\begin{proof}
 The proof can be done by induction on $n$.
\end{proof}

\begin{theorem}
\label{TH0}$P_{n,q}(x)$, satisfies the following identity

\[
D_{q,x}(P_{n,q}(x))=[n]_{q}P_{n-1,q}(x),\ \ n=1,2,...\quad%
\]
\end{theorem}

\begin{proof}
Taking the $q$-derivative of determinant (\ref{4}) with respect to $x$ by
using formula(\ref{5}), given in Lemma \ref{Lemma1} , we obtain%

\begin{equation}
D_{q,x}(P_{n,q}(x))=\frac{(-1)^{n}}{(\beta_{0})^{n+1}}\left\vert
\begin{array}
[c]{ccccc}%
0 & 1 & [2]_{q}x & ... & [n]_{q}x^{n-1}\\
\beta_{0} & \beta_{1} & \beta_{2} & ... & \beta_{n}\\
0 & \beta_{0} & \left[
\begin{array}
[c]{c}%
2\\
1
\end{array}
\right]  _{q}\beta_{1} & ... & \left[
\begin{array}
[c]{c}%
n\\
1
\end{array}
\right]  _{q}\beta_{n-1}\\
0 & 0 & \beta_{0} & ... & \left[
\begin{array}
[c]{c}%
n\\
2
\end{array}
\right]  _{q}\beta_{n-2}\\
\vdots & \vdots &  & \ddots & \vdots\\
\vdots & \vdots &  &  & \vdots\\
0 & 0 & ... & ... & \left[
\begin{array}
[c]{c}%
n\\
n-1
\end{array}
\right]  _{q}\beta_{1}%
\end{array}
\right\vert ,\quad \label{6}%
\end{equation}

Expanding the determinant(\ref{6}) above along with the first column, we have%

\begin{flalign*}
D_{q,x}(P_{n,q}(x))=\frac{(-1)^{n-1}}{(\beta_{0})^{n}}\times && \\
\end{flalign*}
\begin{flalign}
\left\vert
\begin{array}
[c]{ccccccc}%
1 & [2]_{q}x & ... & ... &  & [n-1]_{q}x^{n-2} & [n]_{q}x^{n-1}\\
\beta_{0} & \left[
\begin{array}
[c]{c}%
2\\
1
\end{array}
\right]  _{q}\beta_{1} & ... & ... &  & \left[
\begin{array}
[c]{c}%
n-1\\
1
\end{array}
\right]  _{q}\beta_{n-2} & \left[
\begin{array}
[c]{c}%
n\\
1
\end{array}
\right]  _{q}\beta_{n-1}\\
0 & \beta_{0} & ... & ... &  & \left[
\begin{array}
[c]{c}%
n-1\\
2
\end{array}
\right]  _{q}\beta_{n-3} & \left[
\begin{array}
[c]{c}%
n\\
2
\end{array}
\right]  _{q}\beta_{n-2}\\
\vdots & \vdots & \vdots & \vdots &  & \vdots & \vdots\\
0 & ... & ... &  &  & \beta_{0} & \left[
\begin{array}
[c]{c}%
n\\
n-1
\end{array}
\right]  _{q}\beta_{1}%
\end{array}
\right\vert . \label{7}%
\end{flalign}\\

Now, considering the fact that%

\[
\frac{\lbrack i-1]_{q}}{[j]_{q}}\left[
\begin{array}
[c]{c}%
j\\
i-1
\end{array}
\right]  _{q}=\frac{[i-1]_{q}[j]_{q}!}{[j]_{q}[i-1]_{q}![j-i+1]_{q}}%
=\frac{[j-1]_{q}!}{[i-2]_{q}![j-i+1]_{q}}=\left[
\begin{array}
[c]{c}%
j-1\\
i-2
\end{array}
\right]  _{q},
\]

and multiplying the $j^{th}$ column of the determinant(\ref{7}) by $\frac
{1}{[j]_{q}}$, as well as the $i^{th}$ row by $[i-1]_{q}$ we obtain%

\begin{align*}
D_{q,x}(P_{n,q}(x))=\frac{(-1)^{n-1}}{(\beta_{0})^{n}}\times\frac
{\lbrack1]_{q}!}{[0]_{q}!}\times\frac{\lbrack2]_{q}}{[1]_{q}}\times
...\times\frac{\lbrack n]_{q}}{[n-1]_{q}}\times
\end{align*}
\\
\begin{equation}
\left\vert
\begin{array}
[c]{cccccc}%
1 & x & ... & ... & x^{n-2} & x^{n-1}\\
\beta_{0} & \beta_{1} & ... & ... & \beta_{n-2} & \beta_{n-1}\\
0 & \beta_{0} & ... & ... & \left[
\begin{array}
[c]{c}%
n-2\\
1
\end{array}
\right]  _{q}\beta_{n-3} & \left[
\begin{array}
[c]{c}%
n-1\\
1
\end{array}
\right]  _{q}\beta_{n-2}\\
\vdots &  & \vdots & \vdots & \vdots & \vdots\\
0 & 0 & ... &  & \beta_{0} & \left[
\begin{array}
[c]{c}%
n-1\\
n-2
\end{array}
\right]  _{q}\beta_{1}%
\end{array}
\right\vert , \label{8}%
\end{equation}

which is exactly the desired result.
\end{proof}

\begin{theorem}
\label{TH1}The $q$-polynomials $P_{n,q}(x)$, defined in (\ref{6}), can be
expressed as
\end{theorem}

\begin{equation}
P_{n,q}(x)=\sum_{i=0}^{n}\left[
\begin{array}
[c]{c}%
n\\
j
\end{array}
\right]  _{q}\alpha_{n-j}x^{j}, \label{9}%
\end{equation}

where%

\begin{equation}
\left\{
\begin{array}
[c]{l}%
\alpha_{0}=\frac{1}{\beta_{0}}\\
\alpha_{j}=\frac{(-1)^{j}}{(\beta_{0})^{j+1}}\left\vert
\begin{array}
[c]{ccccccc}%
\beta_{0} & \beta_{1} & \beta_{2} & ... & ... & \beta_{j-1} & \beta_{j}\\
0 & \beta_{0} & \left[
\begin{array}
[c]{c}%
2\\
1
\end{array}
\right]  _{q}\beta_{1} & ... & ... & \left[
\begin{array}
[c]{c}%
j-1\\
1
\end{array}
\right]  _{q}\beta_{j-2} & \left[
\begin{array}
[c]{c}%
j\\
1
\end{array}
\right]  _{q}\beta_{j-1}\\
0 & 0 & \beta_{0} & ... & ... & \left[
\begin{array}
[c]{c}%
j-1\\
2
\end{array}
\right]  _{q}\beta_{j-3} & \left[
\begin{array}
[c]{c}%
j\\
2
\end{array}
\right]  _{q}\beta_{j-2}\\
\vdots & \vdots & \vdots & \ddots &  & \vdots & \vdots\\
\vdots & \vdots & \vdots &  & \ddots & \vdots & \vdots\\
0 & ... & ... & ... & 0 & \beta_{0} & \left[
\begin{array}
[c]{c}%
j\\
j-1
\end{array}
\right]  _{q}\beta_{1}%
\end{array}
\right\vert
\end{array}
\right.. \label{10}%
\end{equation}

\begin{proof}
Expanding the determinant(\ref{4}) along the first row, we obtain%

\begin{align}
P_{n,q}(x)  &  =\frac{(-1)^{n+2}}{(\beta_{0})^{n+1}}\left\vert
\begin{array}
[c]{cccccc}%
\beta_{1} & \beta_{2} & ... & ... & \beta_{n-1} & \beta_{n}\\ \nonumber
\beta_{0} & \left[
\begin{array}
[c]{c}%
2\\
1
\end{array}
\right]  _{q}\beta_{1} & ... & ... & \left[
\begin{array}
[c]{c}%
n-1\\
1
\end{array}
\right]  _{q}\beta_{n-2} & \left[
\begin{array}
[c]{c}%
n\\
1
\end{array}
\right]  _{q}\beta_{n-1}\\
0 & \beta_{0} & ... & ... & \left[
\begin{array}
[c]{c}%
n-1\\
2
\end{array}
\right]  _{q}\beta_{n-3} & \left[
\begin{array}
[c]{c}%
n\\
2
\end{array}
\right]  _{q}\beta_{n-2}\\
\vdots & \vdots & \ddots &  & \vdots & \vdots\\\nonumber
\vdots & \vdots &  & \ddots & \vdots & \vdots\\\nonumber
0 & ... & ... & 0 & \beta_{0} & \left[
\begin{array}
[c]{c}%
n\\
n-1
\end{array}
\right]  _{q}\beta_{1}%
\end{array}
\right\vert \\\nonumber
& +\frac{(-1)^{n+3}}{(\beta_{0})^{n+1}}x\left\vert
\begin{array}
[c]{cccccc}%
\beta_{0} & \beta_{2} & ... & ... & \beta_{n-1} & \beta_{n}\\\nonumber
0 & \left[
\begin{array}
[c]{c}%
2\\
1
\end{array}
\right]  _{q}\beta_{1} & ... & ... & \left[
\begin{array}
[c]{c}%
n-1\\
1
\end{array}
\right]  _{q}\beta_{n-2} & \left[
\begin{array}
[c]{c}%
n\\
1
\end{array}
\right]  _{q}\beta_{n-1}\\\nonumber
0 & \beta_{0} & ... & ... & \left[
\begin{array}
[c]{c}%
n-1\\\nonumber
2
\end{array}
\right]  _{q}\beta_{n-3} & \left[
\begin{array}
[c]{c}%
n\\
2
\end{array}
\right]  _{q}\beta_{n-2}\\
\vdots & \vdots & \ddots &  & \vdots & \vdots\\\nonumber
\vdots & \vdots &  & \ddots & \vdots & \vdots\\\nonumber
0 & ... & ... & 0 & \beta_{0} & \left[
\begin{array}
[c]{c}%
n\\
n-1
\end{array}
\right]  _{q}\beta_{1}%
\end{array}
\right\vert \\
&  +...+\frac{(-1)^{2n+2}}{(\beta_{0})^{n+1}}x^{n}\left\vert
\begin{array}
[c]{cccccc}%
\beta_{0} & \beta_{1} & \beta_{2} & ... & ... & \beta_{n-1}\\
0 & \beta_{0} & \left[
\begin{array}
[c]{c}%
2\\
1
\end{array}
\right]  _{q}\beta_{1} & ... & ... & \left[
\begin{array}
[c]{c}%
n-1\\
1
\end{array}
\right]  _{q}\beta_{n-2}\\
0 & 0 & \beta_{0} & ... & ... & \left[
\begin{array}
[c]{c}%
n-1\\
2
\end{array}
\right]  _{q}\beta_{n-3}\\
\vdots &  &  & \ddots &  & \vdots\\
\vdots &  &  &  & \ddots & \vdots\\
0 & ... & ... & ... & 0 & \beta_{0}%
\end{array}
\right\vert .
\end{align}

Clearly, according to the given definition for $\alpha_{i}$ in (\ref{10}), the
first determinant leads to obtain $\alpha_{n}$, which is the coefficient of
$x^{0}$. Also, the last determinant, which is the determinant of an upper
triangular $n\times n$ matrix, will lead to obtain the coefficient of $x^{n}$
as follows%

\[
\alpha_{0}=\frac{(-1)^{2n+2}}{(\beta_{0})^{n+1}}(\beta_{0})^{n}%
=\frac{1}{\beta_{0}}.
\]

To calculate the coefficient of $x^{j}$ for $0<j<n$, consider the following determinant%

\begin{multline*}
 =\frac{(-1)^{n}}{(\beta_{0})^{n+1}}(-1)^{j+2}\times\\
\left\vert
\begin{array}
[c]{ccccccc}%
\beta_{0} & \beta_{1} & ... & \beta_{j-1} & \beta_{j+1} & ... & \beta_{n}\\
0 & \beta_{0} & ... & \left[
\begin{array}
[c]{c}%
j-1\\
1
\end{array}
\right]  _{q}\beta_{j-2} & \left[
\begin{array}
[c]{c}%
j+1\\
1
\end{array}
\right]  _{q}\beta_{j} & ... & \left[
\begin{array}
[c]{c}%
n\\
1
\end{array}
\right]  _{q}\beta_{n-1}\\
0 & 0 &  & \left[
\begin{array}
[c]{c}%
j-1\\
2
\end{array}
\right]  _{q}\beta_{j-3} & \left[
\begin{array}
[c]{c}%
j+1\\
2
\end{array}
\right]  _{q}\beta_{j-1} & ... & \left[
\begin{array}
[c]{c}%
n\\
2
\end{array}
\right]  _{q}\beta_{n-2}\\
\vdots & \vdots &  & \vdots & \vdots &  & \vdots\\
0 & \ldots &  & \beta_{0} & \left[
\begin{array}
[c]{c}%
j+1\\
j-1
\end{array}
\right]  _{q}\beta_{2} & \ldots & \left[
\begin{array}
[c]{c}%
n\\
2
\end{array}
\right]  _{q}\beta_{n-j-1}\\
\vdots & \ddots &  & 0 & \left[
\begin{array}
[c]{c}%
j+1\\
j
\end{array}
\right]  _{q}\beta_{1} & \ldots & \left[
\begin{array}
[c]{c}%
n\\
j
\end{array}
\right]  _{q}\beta_{n-j}\\
\vdots &  &  & \vdots & \vdots & \ddots & \vdots\\
0 & 0 &  & 0 & 0 & \ldots & \left[
\begin{array}
[c]{c}%
n\\
n-1
\end{array}
\right]  _{q}\beta_{1}%
\end{array}
\right\vert \\
\end{multline*}
\begin{align*}
&  =\frac{(-1)^{n+j}}{(\beta_{0})^{n+1}}(\beta_{0})^{j}\left\vert
\begin{array}
[c]{cccc}%
\left[
\begin{array}
[c]{c}%
j+1\\
j
\end{array}
\right]  _{q}\beta_{1} & \ldots & \left[
\begin{array}
[c]{c}%
n-1\\
j
\end{array}
\right]  _{q}\beta_{n-j-1} & \left[
\begin{array}
[c]{c}%
n\\
j
\end{array}
\right]  _{q}\beta_{n-j}\\
\beta_{0} &  &  & \\
&  &  & \\
\vdots & \ddots & \vdots & \vdots\\
0 & \ldots & \beta_{0} & \left[
\begin{array}
[c]{c}%
n\\
n-1
\end{array}
\right]  _{q}\beta_{1}%
\end{array}
\right\vert .
\end{align*}
Now multiplying the first column of the last determinant by $\frac{1}{\left[
\begin{array}
[c]{c}%
j+1\\
j
\end{array}
\right]  _{q}}$, we obtain%

\[
=\frac{(-1)^{n+j}}{(\beta_{0})^{n-j+1}}\times\frac{1}{\left[
\begin{array}
[c]{c}%
j+1\\
j
\end{array}
\right]  _{q}}\left\vert
\begin{array}
[c]{cccc}%
\beta_{1} & \left[
\begin{array}
[c]{c}%
j+2\\
j
\end{array}
\right]  _{q}\beta_{2} & \ldots & \left[
\begin{array}
[c]{c}%
n\\
j
\end{array}
\right]  _{q}\beta_{n-j}\\
\frac{1}{\left[
\begin{array}
[c]{c}%
j+1\\
j
\end{array}
\right]  _{q}}\beta_{0} & \left[
\begin{array}
[c]{c}%
j+2\\
j+1
\end{array}
\right]  _{q}\beta_{1} & \ldots & \left[
\begin{array}
[c]{c}%
n\\
j+1
\end{array}
\right]  _{q}\beta_{n-j-1}\\
& \beta_{0} &  & \\
&  &  & \\
\vdots &  & \ddots & \vdots\\
0 &  & \ldots & \left[
\begin{array}
[c]{c}%
n\\
n-1
\end{array}
\right]  _{q}\beta_{1}%
\end{array}
\right\vert .
\]

Further similar calculations to get coefficients $1$ for the first elements of each
column in determinant above leads to%

\begin{align*}
&  =\frac{(-1)^{n+j}}{(\beta_{0})^{n-j+1}}\times\frac{1}{\left[
\begin{array}
[c]{c}%
j+1\\
j
\end{array}
\right]  _{q}}\times\frac{1}{\left[
\begin{array}
[c]{c}%
j+2\\
j
\end{array}
\right]  _{q}}\times...\times\frac{1}{\left[
\begin{array}
[c]{c}%
n-1\\
j
\end{array}
\right]  _{q}}\times\frac{1}{\left[
\begin{array}
[c]{c}%
n\\
j
\end{array}
\right]  _{q}}\times\\
&  \left\vert
\begin{array}
[c]{ccccc}%
\beta_{1} & \beta_{2} & \ldots & \beta_{n-j-1} & \beta_{n-j}\\
\frac{1}{\left[
\begin{array}
[c]{c}%
j+1\\
j
\end{array}
\right]  _{q}}\beta_{0} & \frac{\left[
\begin{array}
[c]{c}%
j+2\\
j+1
\end{array}
\right]  _{q}}{\left[
\begin{array}
[c]{c}%
j+2\\
j
\end{array}
\right]  _{q}}\beta_{1} & \ldots & \frac{\left[
\begin{array}
[c]{c}%
n-1\\
j+1
\end{array}
\right]  _{q}}{\left[
\begin{array}
[c]{c}%
n-1\\
j
\end{array}
\right]  _{q}}\beta_{n-j-2} & \frac{\left[
\begin{array}
[c]{c}%
n\\
j+1
\end{array}
\right]  _{q}}{\left[
\begin{array}
[c]{c}%
n\\
j
\end{array}
\right]  _{q}}\beta_{n-j-1}\\
& \beta_{0} &  &  & \\
&  &  &  & \\
\vdots &  & \ddots & \vdots & \vdots\\
0 &  & \ldots & \beta_{0} & \left[
\begin{array}
[c]{c}%
n\\
n-1
\end{array}
\right]  _{q}\beta_{1}%
\end{array}
\right\vert .
\end{align*}
In order to create coefficient $1$ for the term $\beta_{0}$ placed in the
second row of the above determinant, multiply this row by $\left[
\begin{array}
[c]{c}%
j+1\\
j
\end{array}
\right]  $. As we are aware of the fact that%

\[
\frac{\left[
\begin{array}
[c]{c}%
j+2\\
j+1
\end{array}
\right]  _{q}}{\left[
\begin{array}
[c]{c}%
j+2\\
j
\end{array}
\right]  _{q}}.\left[
\begin{array}
[c]{c}%
j+1\\
j
\end{array}
\right]  _{q}=\left[
\begin{array}
[c]{c}%
2\\
1
\end{array}
\right]  _{q},
\]

and also%

\[
\frac{\left[
\begin{array}
[c]{c}%
n\\
j+1
\end{array}
\right]  _{q}}{\left[
\begin{array}
[c]{c}%
n\\
j
\end{array}
\right]  _{q}}\left[
\begin{array}
[c]{c}%
j+1\\
j
\end{array}
\right]  _{q}=\left[
\begin{array}
[c]{c}%
n-j\\
1
\end{array}
\right]  _{q}.
\]

Thus we have%
\begin{align*}
&  =\frac{(-1)^{n+j}}{(\beta_{0})^{n-j+1}}\times\frac{1}{\left[
\begin{array}
[c]{c}%
j+2\\
j
\end{array}
\right]  _{q}}\times...\times\frac{1}{\left[
\begin{array}
[c]{c}%
n-1\\
j
\end{array}
\right]  _{q}}\times\frac{1}{\left[
\begin{array}
[c]{c}%
n\\
j
\end{array}
\right]  _{q}}\\
&  \times\left\vert
\begin{array}
[c]{ccccc}%
\beta_{1} & \beta_{2} & \ldots & \beta_{n-j-1} & \beta_{n-j}\\
\beta_{0} & \left[
\begin{array}
[c]{c}%
2\\
1
\end{array}
\right]  _{q}\beta_{1} & \ldots & \left[
\begin{array}
[c]{c}%
n-j-1\\
1
\end{array}
\right]  _{q}\beta_{n-j-2} & \left[
\begin{array}
[c]{c}%
n-j\\
1
\end{array}
\right]  _{q}\beta_{n-j-1}\\
& \beta_{0} &  &  & \\
&  &  &  & \\
\vdots &  & \ddots & \vdots & \vdots\\
0 &  & \ldots & \beta_{0} & \left[
\begin{array}
[c]{c}%
n\\
n-1
\end{array}
\right]  _{q}\beta_{1}%
\end{array}
\right\vert .
\end{align*}
We continue this method for each row. As the number of coefficients in
\[
\frac{1}{\left[
\begin{array}
[c]{c}%
j+1\\
j
\end{array}
\right]  _{q}}\times\frac{1}{\left[
\begin{array}
[c]{c}%
j+2\\
j
\end{array}
\right]  _{q}}\times...\times\frac{1}{\left[
\begin{array}
[c]{c}%
n-1\\
j
\end{array}
\right]  _{q}}\times\frac{1}{\left[
\begin{array}
[c]{c}%
n\\
j
\end{array}
\right]  _{q}},%
\]
is $n-j$, so it is equal to the number of rows. Moreover, in each step one of
the coefficients above will be cancelled by the corresponding inverse which
will be multiplied later by each row. Therefore, we are sure that at the end
we obtain%

\begin{align*}
&  =\frac{(-1)^{n+j}}{(\beta_{0})^{n-j+1}}\left\vert
\begin{array}
[c]{ccccc}%
\beta_{1} & \beta_{2} & \ldots & \beta_{n-j-1} & \beta_{n-j}\\
\beta_{0} & \left[
\begin{array}
[c]{c}%
2\\
1
\end{array}
\right]  _{q}\beta_{1} & \ldots & \left[
\begin{array}
[c]{c}%
n-j-1\\
1
\end{array}
\right]  _{q}\beta_{n-j-2} & \left[
\begin{array}
[c]{c}%
n-j\\
1
\end{array}
\right]  _{q}\beta_{n-j-1}\\
& \beta_{0} &  &  & \\
&  &  &  & \\
\vdots &  & \ddots & \vdots & \vdots\\
0 &  & \ldots & \beta_{0} & \left[
\begin{array}
[c]{c}%
n\\
n-1
\end{array}
\right]  _{q}\beta_{1}%
\end{array}
\right\vert \\
&  =\alpha_{n-j},
\end{align*}

whence the result.
\end{proof}

\begin{corollary}
The following identity hods for the $q$polynomials $P_{n,q}(x)$%

\begin{equation}
P_{n,q}(x)=\sum_{j=0}^{n}\left[
\begin{array}
[c]{c}%
n\\
j
\end{array}
\right]  _{q}P_{n-j,q}(0)x^{j},\text{ \ \ }n=0,1,2,.... \label{11}%
\end{equation}
\end{corollary}

\begin{proof}
According to the definition(\ref{4}), for $j=0,1,...,n$, $P_{j,q}%
(x)=\alpha_{j}$, since
\end{proof}

\begin{align*}
P_{j,q}(0)  & =\frac{(-1)^{j}}{(\beta_{0})^{j+1}}\times \\
& \left\vert
\begin{array}
[c]{ccccccc}%
1 & 0 & 0 & ... & ... & 0 & 0\\
\beta_{0} & \beta_{1} & \beta_{2} & ... & ... & \beta_{j-1} & \beta_{j}\\
0 & \beta_{0} & \left[
\begin{array}
[c]{c}%
2\\
1
\end{array}
\right]  _{q}\beta_{1} & ... & ... & \left[
\begin{array}
[c]{c}%
j-1\\
1
\end{array}
\right]  _{q}\beta_{j-2} & \left[
\begin{array}
[c]{c}%
j\\
1
\end{array}
\right]  _{q}\beta_{j-1}\\
0 & 0 & \beta_{0} & ... & ... & \left[
\begin{array}
[c]{c}%
j-1\\
2
\end{array}
\right]  _{q}\beta_{j-3} & \left[
\begin{array}
[c]{c}%
j\\
2
\end{array}
\right]  _{q}\beta_{j-2}\\
\vdots &  &  & \ddots &  & \vdots & \vdots\\
\vdots &  &  &  & \ddots & \vdots & \vdots\\
0 & ... & ... & ... & 0 & \beta_{0} & \left[
\begin{array}
[c]{c}%
j\\
j-1
\end{array}
\right]  _{q}\beta_{1}%
\end{array}
\right\vert \text{ \ }\\
\end{align*}
\begin{align*}
&  =\frac{(-1)^{j}}{(\beta_{0})^{j+1}}\left\vert
\begin{array}
[c]{cccccc}%
\beta_{1} & \beta_{2} & ... & ... & \beta_{j-1} & \beta_{j}\\
\beta_{0} & \left[
\begin{array}
[c]{c}%
2\\
1
\end{array}
\right]  _{q}\beta_{1} & ... & ... & \left[
\begin{array}
[c]{c}%
j-1\\
1
\end{array}
\right]  _{q}\beta_{j-2} & \left[
\begin{array}
[c]{c}%
j\\
1
\end{array}
\right]  _{q}\beta_{j-1}\\
0 & \beta_{0} & ... & ... & \left[
\begin{array}
[c]{c}%
j-1\\
2
\end{array}
\right]  _{q}\beta_{j-3} & \left[
\begin{array}
[c]{c}%
j\\
2
\end{array}
\right]  _{q}\beta_{j-2}\\
&  & \ddots &  & \vdots & \vdots\\
&  &  & \ddots & \vdots & \vdots\\
... & ... & ... & 0 & \beta_{0} & \left[
\begin{array}
[c]{c}%
j\\
j-1
\end{array}
\right]  _{q}\beta_{1}%
\end{array}
\right\vert \\
&  =\alpha_{j}\text{.}%
\end{align*}

Replacing $P_{n-j,q}(0)$, instead of $\alpha_{n-j}$ in relation (\ref{9}),
gives the expected result.

\begin{corollary}
The following relations hold for $\alpha_{j}$s in relation (\ref{9})%

\begin{align}
\alpha_{0}  &  =\frac{1}{\beta_{0}},\label{12}\\
\alpha_{j}  &  =-\frac{1}{\beta_{0}}\sum_{i=0}^{j-1}\left[
\begin{array}
[c]{c}%
j\\
i
\end{array}
\right]  _{q}\beta_{j-i}\text{ }\alpha_{i},\text{ \ \ }j=1,2,...,n.\nonumber
\end{align}

\end{corollary}

\begin{proof}
The proof is done by expanding $\alpha_{j}$ , defined in relation(\ref{10}),
along with the first row and also applying a similar technique to the
proof of theorem \ref{TH1}.
\end{proof}

\begin{theorem}
\label{TH2}Suppose that $\{A_{n,q}(x)\}$ be the sequence of $q$-Appell
polynomials with generating function $A_{q}(t)$, defined in the relations
(\ref{2}) and (\ref{3}). If $B_{0,q},$ $B_{1,q},$ $...,$ $B_{n},q,$ with
$B_{0,q}\neq0$ are the coefficients of $q$-Taylor series expansion of the
function $\frac{1}{A_{q}(t)}$ introduced in relation (\ref{0}), then $ for\quad n=0,1,2,...$ we have
\end{theorem}

\begin{equation}
\left\{
\begin{array}
[c]{l}%
A_{0,q}(x)=\frac{1}{B_{0,q}}\\
A_{n,q}(x)=\frac{(-1)^{n}}{(B_{0,q})^{n+1}}\times\\
\left\vert
\begin{array}
[c]{ccccccc}%
1 & x & x^{2} & ... & ... & x^{n-1} & x^{n}\\
B_{0,q} & B_{1,q} & B_{2,q} & ... & ... & B_{n-1,q} & B_{n,q}\\
0 & B_{0,q} & \left[
\begin{array}
[c]{c}%
2\\
1
\end{array}
\right]  _{q}B_{1,q} & ... & ... & \left[
\begin{array}
[c]{c}%
n-1\\
1
\end{array}
\right]  _{q}B_{n-2,q} & \left[
\begin{array}
[c]{c}%
n\\
1
\end{array}
\right]  _{q}B_{n-1,q}\\
0 & 0 & B_{0,q} & ... & ... & \left[
\begin{array}
[c]{c}%
n-1\\
2
\end{array}
\right]  _{q}B_{n-3,q} & \left[
\begin{array}
[c]{c}%
n\\
2
\end{array}
\right]  _{q}B_{n-2,q}\\
\vdots &  &  & \ddots &  & \vdots & \vdots\\
\vdots &  &  &  & \ddots & \vdots & \vdots\\
0 & ... & ... & ... & 0 & B_{0,q} & \left[
\begin{array}
[c]{c}%
n\\
n-1
\end{array}
\right]  _{q}B_{1,q}%
\end{array}
\right\vert
\end{array}
\quad\right.\label{13}%
\end{equation}

\begin{proof}
According to the relations (\ref{2}) and (\ref{3}), we have%

\begin{equation}
A_{q}(t)=\sum_{n=0}^{\infty}A_{n,q}\frac{t^{n}}{\left[  n\right]  _{q}%
!}=A_{0,q}+A_{1,q}t+A_{2,q}\frac{t^{2}}{\left[  2\right]  _{q}!}%
+...+A_{n,q}\frac{t^{n}}{\left[  n\right]  _{q}!}+..., \label{14}%
\end{equation}

and also%

\begin{equation}
A_{q}(t)e_{q}(tx)=\sum_{n=0}^{\infty}A_{n,q}(x)\frac{t^{n}}{\left[  n\right]
_{q}!}=A_{0,q}(x)+A_{1,q}(x)t+A_{2,q}(x)\frac{t^{2}}{\left[  2\right]  _{q}%
!}+...+A_{n,q}(x)\frac{t^{n}}{\left[  n\right]  _{q}!}+.... \label{15}%
\end{equation}
Let $B_{q}(t)=\frac{1}{A_{q}(t)}$. Thus, considering the hypothesis of the
theorem and also noting the definition of $q$-Taylor series expansion of
$B_{q}(t)$ at $a=0$ given in relation (\ref{0}) we have%

\begin{equation}
B_{q}(t)=B_{0,q}+B_{1,q}\frac{t}{\left[  1\right]  _{q}!}+B_{2,q}\frac{t^{2}%
}{\left[  2\right]  _{q}!}+...+B_{n,q}(x)\frac{t^{n}}{\left[  n\right]  _{q}%
!}+.... \label{16}%
\end{equation}

By using Cauchy product rule for the series production $A_{q}(t)B_{q}(t)$, we obtain%

\begin{align*}
1  &  =A_{q}(t)B_{q}(t)\\
&  =\sum_{n=0}^{\infty}A_{n,q}\frac{t^{n}}{\left[  n\right]  _{q}!}\sum
_{n=0}^{\infty}B_{n,q}\frac{t^{n}}{\left[  n\right]  _{q}!}\\
&  =\sum_{n=0}^{\infty}\sum_{k=0}^{n}\left[
\begin{array}
[c]{c}%
n\\
k
\end{array}
\right]  _{q}A_{k,q}B_{n-k,q}\frac{t^{n}}{\left[  n\right]  _{q}!}.
\end{align*}

Consequently,%

\[
\sum_{k=0}^{n}\left[
\begin{array}
[c]{c}%
n\\
k
\end{array}
\right]  _{q}A_{k,q}B_{n-k,q}=\left\{
\begin{array}
[c]{cc}%
1 & \text{ \ \ for }n=0,\\
0 & \text{ \ \ for }n>0.
\end{array}
\right.
\]

It means that%

\begin{equation}
\left\{
\begin{array}
[c]{l}%
B_{0,q}=\frac{1}{A_{0}}\\
B_{n,q}=-\frac{1}{A_{0}}(\sum_{k=1}^{n}\left[
\begin{array}
[c]{c}%
n\\
k
\end{array}
\right]  _{q}A_{k,q}B_{n-k,q}),\text{ \ \ }n=1,2,3,....
\end{array}
\right.  \label{17}%
\end{equation}

Now, multiply both sides of identity(\ref{15}) by $B_{q}(t)=\frac{1}{A_{q}%
(t)}$, and then replace $e_{q}(tx)$ by its $q$-Taylor series expansion, i. e.
$\sum\limits_{k=0}^{\infty}x^{n}\frac{t^{n}}{[n]_{q}!}$. Therefore we obtain%

\begin{align*}
\sum\limits_{k=0}^{\infty}x^{n}\frac{t^{n}}{[n]_{q}!}  &  =e_{q}(tx)\\
&  =B_{q}(t)\sum_{n=0}^{\infty}A_{n,q}(x)\frac{t^{n}}{\left[  n\right]  _{q}%
!}=\sum_{n=0}^{\infty}B_{n,q}\frac{t^{n}}{\left[  n\right]  _{q}!}\sum
_{n=0}^{\infty}A_{n,q}(x)\frac{t^{n}}{\left[  n\right]  _{q}!}.
\end{align*}

Using Cauchy product rule in the last part of relation above leads to%

\begin{equation}
\sum\limits_{k=0}^{\infty}x^{n}\frac{t^{n}}{[n]_{q}!}=\sum_{n=0}^{\infty}%
\sum_{k=0}^{n}\left[
\begin{array}
[c]{c}%
n\\
k
\end{array}
\right]  _{q}B_{n-k,q}A_{k,q}(x)\frac{t^{n}}{\left[  n\right]  _{q}!}.
\label{18}%
\end{equation}

Comparing the coefficients of $\frac{t^{n}}{[n]_{q}!}$ in both sides of
equation(\ref{18}), we have%

\begin{equation}
\sum_{k=0}^{n}\left[
\begin{array}
[c]{c}%
n\\
k
\end{array}
\right]  _{q}B_{n-k,q}A_{k,q}(x)=x^{n},\ \ n=0,1,2,.... \label{19}%
\end{equation}

Writing identity(\ref{19}) for $n=0,1,2,...$ leads to obtain the following
infinite system in the parameter $A_{n,q}(x)$%

\begin{equation}
\left\{
\begin{array}
[c]{l}%
B_{0,q}A_{0,q}(x)=1,\\
B_{1,q}A_{0,q}(x)+B_{0,q}A_{0,q}(x)=x,\\
B_{2,q}A_{0,q}(x)+\left[
\begin{array}
[c]{c}%
2\\
1
\end{array}
\right]  _{q}B_{1,q}A_{1,q}(x)+B_{0,q}A_{2,q}(x)=x^{2},\\
\vdots\\
B_{n,q}A_{0,q}(x)+\left[
\begin{array}
[c]{c}%
n\\
1
\end{array}
\right]  _{q}B_{n-1,q}A_{1,q}(x)+\ldots+B_{0,q}A_{n,q}(x)=x^{n},\\
\vdots.
\end{array}
\right.  \label{20}%
\end{equation}

As it is clear the coefficient matrix of the infinite system (\ref{20}) is
lower triangular. So this property helps us to find $A_{n,q}(x)$ by applying
Cramer rule to only the first $n+1$ equations of this system. Hence we can obtain%

\[
A_{n,q}(x)=\frac{\left\vert
\begin{array}
[c]{cccccc}%
B_{0,q} & 0 & 0 & \cdots & 0 & 1\\
B_{1,q} & B_{0,q} & 0 & \cdots & 0 & x\\
B_{2,q} & \left[
\begin{array}
[c]{c}%
2\\
1
\end{array}
\right]  _{q}B_{1,q} & B_{0,q} & \cdots & 0 & x^{2}\\
\vdots &  &  & \ddots &  & \vdots\\
B_{n-1,q} & \left[
\begin{array}
[c]{c}%
n-1\\
1
\end{array}
\right]  _{q}B_{n-2,q} & \cdots & \cdots & B_{0,q} & x^{n-1}\\
B_{n,q} & \left[
\begin{array}
[c]{c}%
n\\
1
\end{array}
\right]  _{q}B_{n-1,q} & \cdots & \cdots & \left[
\begin{array}
[c]{c}%
n\\
n-1
\end{array}
\right]  _{q}B_{1,q} & x^{n}%
\end{array}
\right\vert }{\left\vert
\begin{array}
[c]{cccccc}%
B_{0,q} & 0 & 0 & \cdots & 0 & 0\\
B_{1,q} & B_{0,q} & 0 & \cdots & 0 & 0\\
B_{2,q} & \left[
\begin{array}
[c]{c}%
2\\
1
\end{array}
\right]  _{q}B_{1,q} & B_{0,q} & \cdots & 0 & 0\\
\vdots &  &  & \ddots &  & \vdots\\
B_{n-1,q} & \left[
\begin{array}
[c]{c}%
n-1\\
1
\end{array}
\right]  _{q}B_{n-2,q} & \cdots & \cdots & B_{0,q} & 0\\
B_{n,q} & \left[
\begin{array}
[c]{c}%
n\\
1
\end{array}
\right]  _{q}B_{n-1,q} & \cdots & \cdots & \left[
\begin{array}
[c]{c}%
n\\
n-1
\end{array}
\right]  _{q}B_{1,q} & B_{0,q}%
\end{array}
\right\vert }%
\]

\[
=\frac{1}{(B_{0,q})^{n+1}}\left\vert
\begin{array}
[c]{cccccc}%
B_{0,q} & 0 & 0 & \cdots & 0 & 1\\
B_{1,q} & B_{0,q} & 0 & \cdots & 0 & x\\
B_{2,q} & \left[
\begin{array}
[c]{c}%
2\\
1
\end{array}
\right]  _{q}B_{1,q} & B_{0,q} & \cdots & 0 & x^{2}\\
\vdots &  &  & \ddots &  & \vdots\\
B_{n-1,q} & \left[
\begin{array}
[c]{c}%
n-1\\
1
\end{array}
\right]  _{q}B_{n-2,q} & \cdots & \cdots & B_{0,q} & x^{n-1}\\
B_{n,q} & \left[
\begin{array}
[c]{c}%
n\\
1
\end{array}
\right]  _{q}B_{n-1,q} & \cdots & \cdots & \left[
\begin{array}
[c]{c}%
n\\
n-1
\end{array}
\right]  _{q}B_{1,q} & x^{n}%
\end{array}
\right\vert
\]

Now, take the transpose of the last determinant and then interchange $i^{th}$
row of the obtained determinant with $i+1^{th}$ \ row, $i=1,2,...,n$. This
leads to obtain the desired result that is exactly relation(\ref{13}).
\end{proof}

\begin{theorem}
\label{TH3}The following facts are equivalent for the $q$-Appell polynomials:
\end{theorem}

\begin{enumerate}
\item[a)] $q$-Appell polynomials can be expressed by considering the relations
(\ref{1}) and (\ref{1'}).

\item[b)] $q$-Appell polynomials can be expressed by considering the relations
(\ref{2}) and (\ref{3}).

\item[c)] $q$-Appell polynomials can be expressed by considering the
determinantal relation (\ref{13}).
\end{enumerate}

\begin{proof}
$(a\Rightarrow b)$ Suppose that relations (\ref{1}) and (\ref{1'}) hold.
Construct an infinite series $\sum_{n=0}^{\infty}A_{n,q}\frac{t^{n}}{\left[
n\right]  _{q}!}$ form all constants $A_{n,q}$ used for defining $A_{n,q}(x)$
in relation (\ref{1'}). Now find the following Cauchy product%

\begin{align*}
&  \sum_{n=0}^{\infty}A_{n,q}\frac{t^{n}}{\left[  n\right]  _{q}!}e_{q}(tx)\\
&  =\sum_{n=0}^{\infty}A_{n,q}\frac{t^{n}}{\left[  n\right]  _{q}!}\sum
_{n=0}^{\infty}x^{n}\frac{t^{n}}{\left[  n\right]  _{q}!}\\
&  \sum_{n=0}^{\infty}\sum_{k=0}^{n}A_{n-k,q}x^{k}\frac{t^{n}}{\left[
n\right]  _{q}!}.
\end{align*}

From relation (\ref{1'}) we know that%

\[
\sum_{k=0}^{n}A_{n-k,q}x^{k}=A_{n,q}(x),
\]

So we find that
\[
\sum_{n=0}^{\infty}A_{n,q}\frac{t^{n}}{\left[  n\right]  _{q}!}e_{q}%
(tx)=A_{n,q}(x),
\]

whence the result.

$(b\Rightarrow c)$ The proof follows directly from Theorem \ref{TH2}.

$(c\Rightarrow a)$ The proof follows from Theorems \ref{TH0} and \ref{TH2}.
\end{proof}

As the consequence of discussion above and particularly Theorem \ref{TH3}, we
are allowed to introduce the determinantal definition of $q$-Appell
polynomials as follows

\begin{definition}
\label{Def1}$q$-Appell polynomials $\{A_{n,q}(x)\}_{n=0}^{\infty}$ can be
defined as
\end{definition}

\begin{equation}
\left\{
\begin{array}
[c]{l}%
A_{0,q}(x)=\frac{1}{B_{0,q}}\\
A_{n,q}(x)=\frac{(-1)^{n}}{(B_{0,q})^{n+1}}\times \\
\left\vert
\begin{array}
[c]{ccccccc}%
1 & x & x^{2} & ... & ... & x^{n-1} & x^{n}\\
B_{0,q} & B_{1,q} & B_{2,q} & ... & ... & B_{n-1,q} & B_{n,q}\\
0 & B_{0,q} & \left[
\begin{array}
[c]{c}%
2\\
1
\end{array}
\right]  _{q}B_{1,q} & ... & ... & \left[
\begin{array}
[c]{c}%
n-1\\
1
\end{array}
\right]  _{q}B_{n-2,q} & \left[
\begin{array}
[c]{c}%
n\\
1
\end{array}
\right]  _{q}B_{n-1,q}\\
0 & 0 & B_{0,q} & ... & ... & \left[
\begin{array}
[c]{c}%
n-1\\
2
\end{array}
\right]  _{q}B_{n-3,q} & \left[
\begin{array}
[c]{c}%
n\\
2
\end{array}
\right]  _{q}B_{n-2,q}\\
\vdots &  &  & \ddots &  & \vdots & \vdots\\
\vdots &  &  &  & \ddots & \vdots & \vdots\\
0 & ... & ... & ... & 0 & B_{0,q} & \left[
\begin{array}
[c]{c}%
n\\
n-1
\end{array}
\right]  _{q}B_{1,q}%
\end{array}
\right\vert
\end{array}
\right.  , \label{21}%
\end{equation}
where $B_{0,q},B_{1,q},B_{2,q},\ldots,B_{n,q}\in\mathbb{R},$ \ $B_{0,q}\neq0$
and $n=1,2,3,....$

\section{Basic Properties of $q$-Appell polynomials from determinantal point
of view}

In this section by using Definition \ref{Def1}, we review the basic properties
of $q$-Appell polynomials.

\begin{theorem}
\label{TH4}For $q$-Appell polynomials the following identities hold
\begin{equation}
A_{n,q}(x)=\frac{1}{B_{0,q}}(x^{n}-\sum_{k=0}^{n-1}\left[
\begin{array}
[c]{c}%
n\\
k
\end{array}
\right]  _{q}B_{n-k,q}A_{k,q}(x)),\text{ \ \ }n=1,2,3,.... \label{22}%
\end{equation}
\end{theorem}
\begin{proof}
Start from expanding the determinant in the Definition \ref{Def1} along with the $n+1^{th}$ row%
\begin{align*}
A_{n,q}(x)=\frac{(-1)^{n}}{(B_{0,q})^{n+1}}\left[
\begin{array}
[c]{c}%
n\\
n-1
\end{array}
\right]  _{q}B_{1,q}\times &&
\end{align*}
\begin{flalign*}
\left\vert
\begin{array}
[c]{cccccc}%
1 & x & x^{2} & ... & ... & x^{n-1}\\
B_{0,q} & B_{1,q} & B_{2,q} & ... & ... & B_{n-1,q}\\
0 & B_{0,q} & \left[
\begin{array}
[c]{c}%
2\\
1
\end{array}
\right]  _{q}B_{1,q} & ... & ... & \left[
\begin{array}
[c]{c}%
n-1\\
1
\end{array}
\right]  _{q}B_{n-2,q}\\
0 & 0 & B_{0,q} & ... & ... & \left[
\begin{array}
[c]{c}%
n-1\\
2
\end{array}
\right]  _{q}B_{n-3,q}\\
\vdots &  &  & \ddots &  & \vdots\\
0 & 0 & \ldots & \ldots & B_{0,q} & \left[
\begin{array}
[c]{c}%
n-1\\
n-2
\end{array}
\right]  _{q}B_{1,q}%
\end{array}
\right\vert
\end{flalign*}
\begin{flalign*}
+\frac{(-1)^{n+1}}{(B_{0,q})^{n+1}}B_{0,q}\times &&
\end{flalign*}
\begin{flalign*}
\left\vert
\begin{array}
[c]{ccccccc}%
1 & x & x^{2} & ... & ... & x^{n-2} & x^{n}\\
B_{0,q} & B_{1,q} & B_{2,q} & ... & ... & B_{n-2,q} & B_{n,q}\\
0 & B_{0,q} & \left[
\begin{array}
[c]{c}%
2\\
1
\end{array}
\right]  _{q}B_{1,q} & ... & ... & \left[
\begin{array}
[c]{c}%
n-2\\
1
\end{array}
\right]  _{q}B_{n-3,q} & \left[
\begin{array}
[c]{c}%
n\\
1
\end{array}
\right]  _{q}B_{n-1,q}\\
0 & 0 & B_{0,q} & ... & ... & \left[
\begin{array}
[c]{c}%
n-2\\
2
\end{array}
\right]  _{q}B_{n-4,q} & \left[
\begin{array}
[c]{c}%
n\\
2
\end{array}
\right]  _{q}B_{n-2,q}\\
\vdots &  &  &  & \ddots & \vdots & \vdots\\
0 & \ldots & \ldots &  &  & B_{0,q} & \left[
\begin{array}
[c]{c}%
n-1\\
n-2
\end{array}
\right]  _{q}B_{2,q}%
\end{array}
\right\vert
\end{flalign*}
\begin{flalign*}
 =\frac{-1}{B_{0,q}}\left[
\begin{array}
[c]{c}%
n\\
n-1
\end{array}
\right]  _{q}B_{1,q}A_{n-1,q}(x)+\frac{(-1)^{n+1}}{(B_{0,q})^{n}}%
\times &&
\end{flalign*}
\begin{flalign*}
\left\vert
\begin{array}
[c]{ccccccc}%
1 & x & x^{2} & ... & ... & x^{n-2} & x^{n}\\
B_{0,q} & B_{1,q} & B_{2,q} & ... & ... & B_{n-2,q} & B_{n,q}\\
0 & B_{0,q} & \left[
\begin{array}
[c]{c}%
2\\
1
\end{array}
\right]  _{q}B_{1,q} & ... & ... & \left[
\begin{array}
[c]{c}%
n-2\\
1
\end{array}
\right]  _{q}B_{n-3,q} & \left[
\begin{array}
[c]{c}%
n\\
1
\end{array}
\right]  _{q}B_{n-1,q}\\
0 & 0 & B_{0,q} & ... & ... & \left[
\begin{array}
[c]{c}%
n-2\\
2
\end{array}
\right]  _{q}B_{n-4,q} & \left[
\begin{array}
[c]{c}%
n\\
2
\end{array}
\right]  _{q}B_{n-2,q}\\
\vdots &  &  &  & \ddots & \vdots & \vdots\\
0 & \ldots & \ldots &  &  & B_{0,q} & \left[
\begin{array}
[c]{c}%
n-1\\
n-2
\end{array}
\right]  _{q}B_{2,q}%
\end{array}
\right\vert
\end{flalign*}

Now repeat the same method for the last determinant%

\begin{flalign*}
&  =\frac{-1}{B_{0,q}}\left[
\begin{array}
[c]{c}%
n\\
n-1
\end{array}
\right]  _{q}B_{1,q}A_{n-1,q}(x)+\frac{(-1)^{n+1}}{(B_{0,q})^{n}}\left[
\begin{array}
[c]{c}%
n-1\\
n-2
\end{array}
\right]  _{q}B_{2,q}\times &&
\end{flalign*}
\begin{flalign*}
\left\vert
\begin{array}
[c]{cccccc}%
1 & x & x^{2} & ... & x^{n-3} & x^{n-2}\\
B_{0,q} & B_{1,q} & B_{2,q} & ... & B_{n-3,q} & B_{n-2,q}\\
0 & B_{0,q} & \left[
\begin{array}
[c]{c}%
2\\
1
\end{array}
\right]  _{q}B_{1,q} & ... & \left[
\begin{array}
[c]{c}%
n-3\\
1
\end{array}
\right]  _{q}B_{n-4,q} & \left[
\begin{array}
[c]{c}%
n-2\\
1
\end{array}
\right]  _{q}B_{n-3,q}\\
0 & 0 & B_{0,q} & ... & \left[
\begin{array}
[c]{c}%
n-3\\
2
\end{array}
\right]  _{q}B_{n-5,q} & \left[
\begin{array}
[c]{c}%
n-2\\
2
\end{array}
\right]  _{q}B_{n-4,q}\\
\vdots &  &  &  & \vdots & \vdots\\
0 & \ldots & \ldots &  & B_{0,q} & \left[
\begin{array}
[c]{c}%
n-2\\
n-3
\end{array}
\right]  _{q}B_{1,q}%
\end{array}
\right\vert
\end{flalign*}
\begin{flalign*}
+\frac{(-1)^{n+2}}{(B_{0,q})^{n}}B_{0,q}\times &&
\end{flalign*}
\begin{flalign*}
=\frac{-1}{B_{0,q}}\left[
\begin{array}
[c]{c}%
n\\
n-1
\end{array}
\right]  _{q}B_{1,q}A_{n-1,q}(x)+\frac{(-1)^{n-1}}{(B_{0,q})^{n}}\left(
\left[
\begin{array}
[c]{c}%
n-1\\
n-2
\end{array}
\right]  _{q}B_{2,q}\frac{(B_{0,q})^{n-1}}{(-1)^{n-2}}A_{n-2,q}(x)\right)
\end{flalign*}
\begin{flalign*}
+\frac{(-1)^{n-2}}{(B_{0,q})^{n-1}}\times \\ &&
\end{flalign*}
\begin{flalign*}
\left\vert
\begin{array}
[c]{ccccccc}%
1 & x & x^{2} & ... & ... & x^{n-3} & x^{n}\\
B_{0,q} & B_{1,q} & B_{2,q} & ... & ... & B_{n-3,q} & B_{n,q}\\
0 & B_{0,q} & \left[
\begin{array}
[c]{c}%
2\\
1
\end{array}
\right]  _{q}B_{1,q} & ... & ... & \left[
\begin{array}
[c]{c}%
n-3\\
1
\end{array}
\right]  _{q}B_{n-4,q} & \left[
\begin{array}
[c]{c}%
n\\
1
\end{array}
\right]  _{q}B_{n-1,q}\\
0 & 0 & B_{0,q} & ... & ... & \left[
\begin{array}
[c]{c}%
n-3\\
2
\end{array}
\right]  _{q}B_{n-5,q} & \left[
\begin{array}
[c]{c}%
n\\
2
\end{array}
\right]  _{q}B_{n-2,q}\\
\vdots &  &  &  & \ddots & \vdots & \vdots\\
0 & \ldots & \ldots &  & 0 & B_{0,q} & \left[
\begin{array}
[c]{c}%
n-1\\
n-2
\end{array}
\right]  _{q}B_{2,q}%
\end{array}
\right\vert
\end{flalign*}
\begin{flalign*}
=\frac{-1}{B_{0,q}}\left[
\begin{array}
[c]{c}%
n\\
n-1
\end{array}
\right]  _{q}B_{1,q}A_{n-1,q}(x)-\frac{1}{B_{0,q}}\left[
\begin{array}
[c]{c}%
n-1\\
n-2
\end{array}
\right]  _{q}B_{2,q}A_{n-2,q}(x)+\frac{(-1)^{n-2}}{(B_{0,q})^{n-1}}\times &&
\end{flalign*}
\begin{flalign*}
\left\vert
\begin{array}
[c]{ccccccc}%
1 & x & x^{2} & ... & ... & x^{n-3} & x^{n}\\
B_{0,q} & B_{1,q} & B_{2,q} & ... & ... & B_{n-3,q} & B_{n,q}\\
0 & B_{0,q} & \left[
\begin{array}
[c]{c}%
2\\
1
\end{array}
\right]  _{q}B_{1,q} & ... & ... & \left[
\begin{array}
[c]{c}%
n-3\\
1
\end{array}
\right]  _{q}B_{n-4,q} & \left[
\begin{array}
[c]{c}%
n\\
1
\end{array}
\right]  _{q}B_{n-1,q}\\
0 & 0 & B_{0,q} & ... & ... & \left[
\begin{array}
[c]{c}%
n-3\\
2
\end{array}
\right]  _{q}B_{n-5,q} & \left[
\begin{array}
[c]{c}%
n\\
2
\end{array}
\right]  _{q}B_{n-2,q}\\
\vdots &  &  &  & \ddots & \vdots & \vdots\\
0 & \ldots & \ldots &  & 0 & B_{0,q} & \left[
\begin{array}
[c]{c}%
n-1\\
n-2
\end{array}
\right]  _{q}B_{2,q}%
\end{array}
\right\vert
\end{flalign*}

Continue a similar method to arrive at%

\begin{equation*}
=\frac{-1}{B_{0,q}}\left[
\begin{array}
[c]{c}%
n\\
n-1
\end{array}
\right]  _{q}B_{1,q}A_{n-1,q}(x)-\frac{1}{B_{0,q}}\left[
\begin{array}
[c]{c}%
n-1\\
n-2
\end{array}
\right]  _{q}B_{2,q}A_{n-2,q}(x)\\
\end{equation*}
\begin{equation*}
-\ldots-\frac{1}{(B_{0,q})^{2}}\left\vert
\begin{array}
[c]{cc}%
1 & x^{n}\\
B_{0,q} & B_{n,q}%
\end{array}
\right\vert \\
\end{equation*}
\begin{equation*}
=\frac{-1}{B_{0,q}}\left[
\begin{array}
[c]{c}%
n\\
n-1
\end{array}
\right]  _{q}B_{1,q}A_{n-1,q}(x)-\frac{1}{B_{0,q}}\left[
\begin{array}
[c]{c}%
n-1\\
n-2
\end{array}
\right]  _{q}B_{2,q}A_{n-2,q}(x)\\
\end{equation*}
\begin{equation*}
-\ldots-\frac{1}{(B_{0,q})^{2}}\left(
B_{n,q}-B_{0,q}x^{n}\right)  \\
\end{equation*}
\begin{equation*}
=\frac{-1}{B_{0,q}}\left[
\begin{array}
[c]{c}%
n\\
n-1
\end{array}
\right]  _{q}B_{1,q}A_{n-1,q}(x)-\frac{1}{B_{0,q}}\left[
\begin{array}
[c]{c}%
n-1\\
n-2
\end{array}
\right]  _{q}B_{2,q}A_{n-2,q}(x)-\ldots\\
\end{equation*}
\begin{equation*}
-\frac{1}{B_{0,q}}B_{n,q}%
A_{0,q}(x)+\frac{1}{B_{0,q}}x^{n}\\
\end{equation*}
\begin{equation*}
=\frac{1}{B_{0,q}}(x^{n}-\sum_{k=0}^{n-1}\left[
\begin{array}
[c]{c}%
n\\
k
\end{array}
\right]  _{q}B_{n-k,q}A_{k,q}(x)).
\end{equation*}

\end{proof}
\begin{corollary}
Powers of $x$ can be expressed based on $q$-Appell polynomials as
\end{corollary}

\begin{equation}
x^{n}=\sum_{k=0}^{n}\left[
\begin{array}
[c]{c}%
n\\
k
\end{array}
\right]  _{q}B_{n-k,q}A_{k,q}(x),\ \ n=1,2,3,.... \label{23}%
\end{equation}

\begin{proof}
The proof is the direct result of relation(\ref{22}) in Theorem \ref{TH4}.
\end{proof}

\begin{notation}
Suppose $P_{n}(x)$ and $Q_{n}(x)$ are two polynomials of degree $n$. Let
$P_{n}(x)$ be defined as in relation(\ref{4}). Then for n=1,2,3,..., we have
\begin{flalign*}
(PQ)(x):=\frac{(-1)^{n}}{(\beta_{0})^{n+1}}\times\\ &&
\end{flalign*}
\begin{equation}
\left\vert
\begin{array}
[c]{ccccccc}%
Q_{0}(x) & Q_{1}(x) & Q_{2}(x) & ... & ... & Q_{n-1}(x) & Q_{n}(x)\\
\beta_{0} & \beta_{1} & \beta_{2} & ... & ... & \beta_{n-1} & \beta_{n}\\
0 & \beta_{0} & \left[
\begin{array}
[c]{c}%
2\\
1
\end{array}
\right]  _{q}\beta_{1} & ... & ... & \left[
\begin{array}
[c]{c}%
n-1\\
1
\end{array}
\right]  _{q}\beta_{n-2} & \left[
\begin{array}
[c]{c}%
n\\
1
\end{array}
\right]  _{q}\beta_{n-1}\\
0 & 0 & \beta_{0} & ... & ... & \left[
\begin{array}
[c]{c}%
n-1\\
2
\end{array}
\right]  _{q}\beta_{n-3} & \left[
\begin{array}
[c]{c}%
n\\
2
\end{array}
\right]  _{q}\beta_{n-2}\\
\vdots &  &  & \ddots &  & \vdots & \vdots\\
\vdots &  &  &  & \ddots & \vdots & \vdots\\
0 & ... & ... & ... & 0 & \beta_{0} & \left[
\begin{array}
[c]{c}%
n\\
n-1
\end{array}
\right]  _{q}\beta_{1}%
\end{array}
\right\vert \text{ \ \ },\text{ } \label{24}%
\end{equation}
\end{notation}

\begin{theorem}
Suppose that $\{A_{n,q}(x)\}_{n=0}^{\infty}$ and $\{\tilde{A}_{n,q}%
(x)\}_{n=0}^{\infty}$ are two families of $q$-Appell polynomials. Then
\end{theorem}

\begin{enumerate}
\item[a)] For every $\alpha$ and $\beta\in\mathbb{R}$, $\{\alpha
A_{n,q}(x)+\beta\tilde{A}_{n,q}(x)\}_{n=0}^{\infty}$ is also a family of
$q$-Appell polynomials.

\item[b)] $\{(A\tilde{A})_{n,q}(x)\}_{n=0}^{\infty}$ is also a family of
$q$-Appell polynomials.
\end{enumerate}

\begin{proof}
a) The proof is the direct consequence of linear properties of determinant.

b) According to the determinantal definition of $q$-Appell polynomials given
in Theorem \ref{TH2} relation(\ref{13}) and also notation(\ref{24}), we have%

\begin{flalign*}
(A\tilde{A})_{n,q}(x)  &  =A_{n,q}(\tilde{A}_{n,q}(x))=\\
\frac{(-1)^{n}}{(B_{0,q})^{n+1}}\times &&
\end{flalign*}
\begin{align*}
\left\vert
\begin{array}
[c]{ccccccc}%
\tilde{A}_{0,q}(x) & \tilde{A}_{1,q}(x) & \tilde{A}_{2,q}(x) & ... & ... &
\tilde{A}_{n-1,q}(x) & \tilde{A}_{n,q}(x)\\
B_{0,q} & B_{1,q} & B_{2,q} & ... & ... & B_{n-1,q} & B_{n,q}\\
0 & B_{0,q} & \left[
\begin{array}
[c]{c}%
2\\
1
\end{array}
\right]  _{q}B_{1,q} & ... & ... & \left[
\begin{array}
[c]{c}%
n-1\\
1
\end{array}
\right]  _{q}B_{n-2,q} & \left[
\begin{array}
[c]{c}%
n\\
1
\end{array}
\right]  _{q}B_{n-1,q}\\
0 & 0 & B_{0,q} & ... & ... & \left[
\begin{array}
[c]{c}%
n-1\\
2
\end{array}
\right]  _{q}B_{n-3,q} & \left[
\begin{array}
[c]{c}%
n\\
2
\end{array}
\right]  _{q}B_{n-2,q}\\
\vdots &  &  & \ddots &  & \vdots & \vdots\\
\vdots &  &  &  & \ddots & \vdots & \vdots\\
0 & ... & ... & ... & 0 & B_{0,q} & \left[
\begin{array}
[c]{c}%
n\\
n-1
\end{array}
\right]  _{q}B_{1,q}%
\end{array}
\right\vert .
\end{align*}

Using formula(\ref{5}) given in Lemma \ref{Lemma1} we have%

\begin{flalign*}
D_{q}((A\tilde{A})_{n,q}(x))\\
=\frac{(-1)^{n}}{(B_{0,q})^{n+1}}\times \\ &&
\end{flalign*}
\begin{align*}
\left\vert
\begin{array}
[c]{ccccccc}%
D_{q}(\tilde{A}_{0,q}(x)) & D_{q}(\tilde{A}_{1,q}(x)) & D_{q}(\tilde{A}%
_{2,q}(x)) & ... & ... & D_{q}(\tilde{A}_{n-1,q}(x)) & D_{q}(\tilde{A}%
_{n,q}(x))\\
B_{0,q} & B_{1,q} & B_{2,q} & ... & ... & B_{n-1,q} & B_{n,q}\\
0 & B_{0,q} & \left[
\begin{array}
[c]{c}%
2\\
1
\end{array}
\right]  _{q}B_{1,q} & ... & ... & \left[
\begin{array}
[c]{c}%
n-1\\
1
\end{array}
\right]  _{q}B_{n-2,q} & \left[
\begin{array}
[c]{c}%
n\\
1
\end{array}
\right]  _{q}B_{n-1,q}\\
0 & 0 & B_{0,q} & ... & ... & \left[
\begin{array}
[c]{c}%
n-1\\
2
\end{array}
\right]  _{q}B_{n-3,q} & \left[
\begin{array}
[c]{c}%
n\\
2
\end{array}
\right]  _{q}B_{n-2,q}\\
\vdots &  &  & \ddots &  & \vdots & \vdots\\
\vdots &  &  &  & \ddots & \vdots & \vdots\\
0 & ... & ... & ... & 0 & B_{0,q} & \left[
\begin{array}
[c]{c}%
n\\
n-1
\end{array}
\right]  _{q}B_{1,q}%
\end{array}
\right\vert
\end{align*}

Since $\{\tilde{A}_{n,q}(x)\}_{n=0}^{\infty}$ is a family of $q$-Appell
polynomials, according to relation(\ref{1}) we have%

\[
D_{q,x}(\tilde{A}_{n,q}(x))=[n]_{q}\tilde{A}_{n-1,q}(x),\ \ n=0,1,2,....
\]

Therefore we can continue as%

\begin{flalign*}
D_{q}((A\tilde{A})_{n,q}(x))= \frac{(-1)^{n}}{(B_{0,q})^{n+1}}\times \\ &&
\end{flalign*}
\begin{align*}
\left\vert
\begin{array}
[c]{ccccccc}%
0 & \tilde{A}_{0,q}(x) & [2]_{q}\tilde{A}_{1,q}(x) & ... & ... &
[n-1]_{q}\tilde{A}_{n-2,q}(x) & [n]_{q}\tilde{A}_{n-1,q}(x)\\
B_{0,q} & B_{1,q} & B_{2,q} & ... & ... & B_{n-1,q} & B_{n,q}\\
0 & B_{0,q} & \left[
\begin{array}
[c]{c}%
2\\
1
\end{array}
\right]  _{q}B_{1,q} & ... & ... & \left[
\begin{array}
[c]{c}%
n-1\\
1
\end{array}
\right]  _{q}B_{n-2,q} & \left[
\begin{array}
[c]{c}%
n\\
1
\end{array}
\right]  _{q}B_{n-1,q}\\
0 & 0 & B_{0,q} & ... & ... & \left[
\begin{array}
[c]{c}%
n-1\\
2
\end{array}
\right]  _{q}B_{n-3,q} & \left[
\begin{array}
[c]{c}%
n\\
2
\end{array}
\right]  _{q}B_{n-2,q}\\
\vdots &  &  & \ddots &  & \vdots & \vdots\\
\vdots &  &  &  & \ddots & \vdots & \vdots\\
0 & ... & ... & ... & 0 & B_{0,q} & \left[
\begin{array}
[c]{c}%
n\\
n-1
\end{array}
\right]  _{q}B_{1,q}%
\end{array}
\right\vert .
\end{align*}

Now, expand the last determinant along with the first column as follows%

\begin{flalign*}
&  =\frac{(-1)^{n}}{(B_{0,q})^{n+1}}\times-B_{0,q}\times  &&
\end{flalign*}
\begin{flalign*}
\left\vert
\begin{array}
[c]{cccccc}%
\tilde{A}_{0,q}(x) & [2]_{q}\tilde{A}_{1,q}(x) & ... & ... & [n-1]_{q}%
\tilde{A}_{n-2,q}(x) & [n]_{q}\tilde{A}_{n-1,q}(x)\\
B_{0,q} & \left[
\begin{array}
[c]{c}%
2\\
1
\end{array}
\right]  _{q}B_{1,q} & ... & ... & \left[
\begin{array}
[c]{c}%
n-1\\
1
\end{array}
\right]  _{q}B_{n-2,q} & \left[
\begin{array}
[c]{c}%
n\\
1
\end{array}
\right]  _{q}B_{n-1,q}\\
0 & B_{0,q} & ... & ... & \left[
\begin{array}
[c]{c}%
n-1\\
2
\end{array}
\right]  _{q}B_{n-3,q} & \left[
\begin{array}
[c]{c}%
n\\
2
\end{array}
\right]  _{q}B_{n-2,q}\\
\vdots &  & \ddots &  & \vdots & \vdots\\
\vdots &  &  & \ddots & \vdots & \vdots\\
... & ... & ... & 0 & B_{0,q} & \left[
\begin{array}
[c]{c}%
n\\
n-1
\end{array}
\right]  _{q}B_{1,q}%
\end{array}
\right\vert
\end{flalign*}
\begin{flalign*}
=[n]_{q}(A\tilde{A})_{n-1,q}(x),&&
\end{flalign*}

which means that $\{(A\tilde{A})_{n,q}(x)\}_{n=0}^{\infty}$ belongs to the
family of $q$-Appell polynomials too.
\end{proof}

\begin{definition}
\label{2D}We define 2D $q$-Appell polynomials $\{A_{n,q}(x,y)\}_{n=0}^{\infty
}$ by means of the generating function below%
\begin{equation}
A_{q}(x,y,t):=A_{q}(t)e_{q}(tx)E_{q}(ty)=\sum\limits_{n=0}^{\infty}%
A_{n,q}(x,y)\frac{t^{n}}{\left[  n\right]  _{q}!}, \label{24'}%
\end{equation}
or equivalently%
\begin{equation}
\left\{
\begin{array}
[c]{l}%
A_{0,q}(x,y)=\frac{1}{B_{0,q}}\\
A_{n,q}(x,y)=\frac{(-1)^{n}}{(B_{0,q})^{n+1}}\times \\
\left\vert
\begin{array}
[c]{ccccccc}%
1 & x+y & (x+y)_{q}^{2} & ... & ... & (x+y)_{q}^{n-1} & (x+y)_{q}^{n}\\
B_{0,q} & B_{1,q} & B_{2,q} & ... & ... & B_{n-1,q} & B_{n,q}\\
0 & B_{0,q} & \left[
\begin{array}
[c]{c}%
2\\
1
\end{array}
\right]  _{q}B_{1,q} & ... & ... & \left[
\begin{array}
[c]{c}%
n-1\\
1
\end{array}
\right]  _{q}B_{n-2,q} & \left[
\begin{array}
[c]{c}%
n\\
1
\end{array}
\right]  _{q}B_{n-1,q}\\
0 & 0 & B_{0,q} & ... & ... & \left[
\begin{array}
[c]{c}%
n-1\\
2
\end{array}
\right]  _{q}B_{n-3,q} & \left[
\begin{array}
[c]{c}%
n\\
2
\end{array}
\right]  _{q}B_{n-2,q}\\
\vdots &  &  & \ddots &  & \vdots & \vdots\\
\vdots &  &  &  & \ddots & \vdots & \vdots\\
0 & ... & ... & ... & 0 & B_{0,q} & \left[
\begin{array}
[c]{c}%
n\\
n-1
\end{array}
\right]  _{q}B_{1,q}%
\end{array}
\right\vert
\end{array}
\text{.}\right.  \label{24''}%
\end{equation}

\end{definition}

\begin{remark}
From the Definition \ref{2D}, it is clear that%

\begin{equation}
A_{n,q}(x,0)=A_{n,q}(x). \label{24'''}%
\end{equation}

\end{remark}

\begin{theorem}
The following fact holds for 2D $q$-Appell polynomials $\{A_{n,q}%
(x,y)\}_{n=0}^{\infty}$%

\begin{equation}
A_{n,q}(x,y)=\sum\limits_{k=0}^{n}\left[
\begin{array}
[c]{c}%
n\\
k
\end{array}
\right]  _{q}q^{1/2(n-k)(n-k-1)}A_{k,q}(x)y^{n-k}. \label{25}%
\end{equation}

\end{theorem}

\begin{proof}
Proof is simple and based on properties of determinant.
\end{proof}

\begin{corollary}
The following difference identity holds for $q$-Appell polynomials
$\{A_{n,q}(x)\}_{n=0}^{\infty}$%
\begin{equation}
A_{n,q}(x,1)-A_{n,q}(x)=\sum\limits_{k=1}^{n-1}\left[
\begin{array}
[c]{c}%
n\\
k
\end{array}
\right]  _{q}q^{1/2(n-k)(n-k-1)}A_{k,q}(x),\text{ \ \ }n=0,1,2,.... \label{26}%
\end{equation}

\end{corollary}

\begin{proof}
Using relations (\ref{24'''}) and also (\ref{25}) for $y=1$ and $y=0$ and
replacing the results in the left side of relation(\ref{26}) leads to reach to
the right side of this relation.
\end{proof}

\begin{theorem}
\label{TH5}For every $t\in\mathbb{R}$, the following facts are equivalent for
$q$-Appell polynomials $\{A_{n,q}(x)\}_{n=0}^{\infty}$

\begin{enumerate}
\item[a)] $A_{n,q}(x,-y)=(-1)^{n}A_{n,q}(0,y),$

\item[b)] $A_{n,q}(x)=(-1)^{n}A_{n,q}(0).$
\end{enumerate}
\end{theorem}

\begin{proof}
$(a\Rightarrow b)$ The proof is done using part (a) for $x=0$.

$(b\Rightarrow a)$ We apply the relation(\ref{25}) for the left hand side of
part (a) as follows%
\begin{align*}
A_{n,q}(x,-y)  &  =\sum\limits_{k=0}^{n}\left[
\begin{array}
[c]{c}%
n\\
k
\end{array}
\right]  _{q}q^{1/2(n-k)(n-k-1)}A_{k,q}(x,0)(-y)^{n-k}\\
&  =(-1)^{n}\sum\limits_{k=0}^{n}\left[
\begin{array}
[c]{c}%
n\\
k
\end{array}
\right]  _{q}q^{1/2(n-k)(n-k-1)}A_{k,q}(x,0)(-1)^{k}y^{n-k}\\
&  =(-1)^{n}\sum\limits_{k=0}^{n}\left[
\begin{array}
[c]{c}%
n\\
k
\end{array}
\right]  _{q}q^{\frac{k(k-1)}{2}}A_{n-k,q}(x,0)(-1)^{n-k}y^{k}.
\end{align*}

Using part (b), we have%

\[
A_{n,q}(x,-y)=(-1)^{n}\sum\limits_{k=0}^{n}\left[
\begin{array}
[c]{c}%
n\\
k
\end{array}
\right]  _{q}q^{\frac{k(k-1)}{2}}A_{n-k,q}(0)x^{k}.
\]

Now, using Definition \ref{2D} leads to obtain%

\[
A_{n,q}(x,-y)=(-1)^{n}A_{n-k,q}(0,y),
\]

whence the result.
\end{proof}

\begin{lemma}
\label{Lemma2}In relation(\ref{17}) for the coefficients $A_{n,q}$ and
$B_{n,q}$ we have
\end{lemma}

\begin{theorem}%
\begin{equation}
A_{2n+1,q}=0\Leftrightarrow B_{2n+1,q}=0,\ \ n=0,1,2,.... \label{27}%
\end{equation}

\end{theorem}

\begin{proof}
$(\Rightarrow)$We have already known the following fact from relation(\ref{17}) for $n=0,1,2,...$
\end{proof}

$\left\{
\protect\begin{array}
[c]{l}%
B_{1,q}=-\frac{1}{A_{0}}A_{1,q}B_{0,q},\\
B_{2n+1,q}=-\frac{1}{A_{0}}\left[
\protect\begin{array}
[c]{c}%
2n+1\\
k
\protect\end{array}
\right]  _{q}A_{1,q}B_{2n,q}\\
+-\frac{1}{A_{0}}\left(  \sum_{k=1}^{n}\left(  \left[
\protect\begin{array}
[c]{c}%
2n+1\\
2k
\protect\end{array}
\right]  _{q}A_{2k,q}B_{2n-k+1,q}+\left[
\protect\begin{array}
[c]{c}%
2n+1\\
2k+1
\protect\end{array}
\right]  _{q}A_{2k+1,q}B_{2n-k+1,q}\right)  \right)  .
\protect\end{array}
\right.  $
\\

Since $A_{2n+1,q}=0$ \ \ for $n=0,1,2,...,$ then%

\[
\left\{
\begin{array}
[c]{l}%
B_{1,q}=0\\
B_{2n+1,q}=-\frac{1}{A_{0}}\sum_{k=1}^{n}\left[
\begin{array}
[c]{c}%
2n+1\\
2k
\end{array}
\right]  _{q}A_{2k,q}B_{2n-k+1,q},\text{ \ \ }n=1,2,3,....
\end{array}
\right.
\]

Consequently, we should have $B_{2n+1,q}=0,$ \ \ for $n=0,1,2,....$

$(\Leftarrow)$ In a similar way to the above we can prove it.

\begin{theorem}
The following facts are equivalent for $q$-Appell polynomials $\{A_{n,q}%
(x)\}_{n=0}^{\infty}$

\begin{enumerate}
\item[a)] $A_{n,q}(-x)=(-1)^{n}A_{n,q}(x),$

\item[b)] $B_{2n+1,q}=0,$ \ \ for $n=0,1,2,....$
\end{enumerate}
\end{theorem}

\begin{proof}
According to Theorem\ref{TH5}, we know that%

\[
A_{n,q}(-x)=(-1)^{n}A_{n,q}(x)\Leftrightarrow A_{n,q}(t)=(-1)^{n}A_{n,q}(0)
\]

So using Lemma\ref{Lemma2}, we have%

\[
\Leftrightarrow A_{2n+1,q}(0)=(-1)^{n}A_{2n+1,q}(0)\Leftrightarrow
A_{2n+1,q}=0\Leftrightarrow B_{2n+1,q}=0.
\]

\end{proof}

\begin{theorem}
For every $n\geq1$, $q$-Appell polynomials $\{A_{n,q}(x)\}_{n=0}^{\infty}$
satisfy the following identities
\end{theorem}

\begin{equation}
\int\limits_{0}^{x}A_{n,q}(t)d_{q}t=\frac{1}{[n+1]_{q}}\left(  A_{n+1,q}%
(x)-A_{n,q}(0)\right)  \label{28}%
\end{equation}

\begin{equation}
\int\limits_{0}^{x}A_{n,q}(t)d_{q}t=\frac{1}{[n+1]_{q}}\sum\limits_{k=0}%
^{n+1}\left[
\begin{array}
[c]{c}%
n+1\\
k
\end{array}
\right]  _{q}q^{1/2k(k-1)}A_{n-k,q}(0) \label{29}%
\end{equation}

\begin{proof}
Relation(\ref{28}) is the direct result of property(\ref{1}) for $q$-Appell
polynomials $\{A_{n,q}(x)\}_{n=0}^{\infty}$. To prove equality(\ref{29}), we
start from relation(\ref{28}) for $x=1$ as follows%

\[
\int\limits_{0}^{1}A_{n,q}(t)d_{q}t=\frac{1}{[n+1]_{q}}\left(  A_{n+1,q}%
(1)-A_{n,q}(0)\right)  .
\]

Now, find $A_{n+1,q}(1)$ using relation(\ref{25}) by assuming $x=0$ and $y=1$%

\[
A_{n+1,q}(1)=\sum\limits_{k=0}^{n+1}\left[
\begin{array}
[c]{c}%
n+1\\
k
\end{array}
\right]  _{q}q^{1/2k(k-1)}A_{n-k,q}(0).
\]

Therefore, we obtain%

\[
\int\limits_{0}^{1}A_{n,q}(t)d_{q}t=\frac{1}{[n+1]_{q}}\sum\limits_{k=0}%
^{n}\left[
\begin{array}
[c]{c}%
n+1\\
k
\end{array}
\right]  _{q}q^{1/2k(k-1)}A_{n-k,q}(0).
\]

\end{proof}


                                                                                             %





\end{document}